\documentclass[12pt]{article}
\usepackage{bez123,calc,curves,ebezier,epic,eepic,graphicx,multiply,rotating}
\everymath{\displaystyle}
\usepackage{graphicx}
\usepackage{pst-all}
\usepackage{color}
\usepackage{amssymb}
\usepackage{amsmath}
\newtheorem{prethm}{{\bf Theorem}}

\newenvironment{thm}{\begin{prethm}{\hspace{-0.5
               em}{\bf .}}}{\end{prethm}}
\newtheorem{prelemma}{{\bf Lemma}}

\newenvironment{lemma}{\begin{prelemma}{\hspace{-0.5
               em}{\bf .}}}{\end{prelemma}}
\newtheorem{preex}{{\bf Example}}

\newtheorem{preprop}{{\bf Proposition}}

\newenvironment{prop}{\begin{preprop}{\hspace{-0.5em}{\bf .}}}{\end{preprop}}
\newtheorem{precor}{{\bf Corollary}}

\newtheorem{preremark}{{\bf Remark}}

\newenvironment{remark}{\begin{preremark}{\hspace{-0.5
               em}{\bf.}}}{\end{preremark}}
\newtheorem{preprob}{{\bf Problem}}

\newtheorem{predefin}{{\bf Definition}}

\newenvironment{defin}{\begin{predefin}{\hspace{-0.5
               em}{\bf .}}}{\end{predefin}}
\newtheorem{preconj}{{\bf Conjecture}}

\newtheorem{preprobb}{{\bf Problem}}

\newtheorem{prelem}{{\bf Theorem}}

\newenvironment{proof}{{\bf Proof.}\rm }{\hfill{$\Box$}}

\newtheorem{presolution}{{\bf Solution.}}

\def\newpic#1{}
\def\qed{\ifhmode\unskip\nobreak\fi\quad\ifmmode\Box\else$\Box$\fi}

\title{\Large\bf\noindent On the largest dynamic monopolies of graphs with a given average threshold}

\author{\large\bf Kaveh Khoshkhah~~~~~Manouchehr Zaker\footnote{E-mail: mzaker@iasbs.ac.ir}
\vspace{5mm}\\
    Department of Mathematics,\\
     Institute for Advanced Studies in Basic Sciences,\\
    Zanjan 45137-66731, Iran}
    \date{}
\begin{document}
\maketitle
\begin{abstract}
\noindent Let $G$ be a graph and $\tau$ be an assignment of nonnegative integer thresholds to the vertices of $G$. A subset of vertices $D$ is said to be a $\tau$-dynamic monopoly, if $V(G)$ can be partitioned into subsets $D_0, D_1, \ldots, D_k$ such that $D_0=D$ and for any $i\in \{0, \ldots, k-1\}$, each vertex $v$ in $D_{i+1}$ has at least $\tau(v)$ neighbors in $D_0\cup \ldots \cup D_i$. Denote the size of smallest $\tau$-dynamic monopoly by $dyn_{\tau}(G)$ and the average of thresholds in $\tau$ by $\overline{\tau}$. We show that the values of $dyn_{\tau}(G)$ over all assignments $\tau$ with the same average threshold is a continuous set of integers. For any positive number $t$, denote the maximum $dyn_{\tau}(G)$ taken over all threshold assignments $\tau$ with $\overline{\tau}\leq t$, by $Ldyn_t(G)$. In fact, $Ldyn_t(G)$ shows the worst-case value of a dynamic monopoly when the average threshold is a given number $t$. We investigate under what conditions on $t$, there exists an upper bound for $Ldyn_{t}(G)$ of the form $c|G|$, where $c<1$. Next, we show that $Ldyn_t(G)$ is coNP-hard for planar graphs but has polynomial-time solution for forests.
\end{abstract}




\section{Introduction}

\noindent In this paper we deal with simple undirected graphs. For any such graph $G=(V,E)$, we denote the cardinality of its vertex set by $|G|$ and the edge density of graph $G$ by $\epsilon(G):=|E|/|G|$. We denote the degree of a vertex $v$ in $G$ by $deg_G(v)$. For other graph theoretical notations we refer the reader to \cite{BM}. By a threshold assignment for the vertices of $G$ we mean any function $\tau: V(G) \rightarrow \Bbb{N}\cup \{0\}$. A subset of vertices $D$ is said to be a {\it $\tau$-dynamic monopoly} of $G$ or simply {\it $\tau$-dynamo} of $G$, if for some nonnegative integer $k$, the vertices of $G$ can be partitioned into subsets $D_0, D_1, \ldots, D_k$ such that $D_0=D$ and for any $i$, $1\leq i \leq k$, the set $D_{i}$ consists of all vertices $v$ which has at least $\tau(v)$ neighbors in $D_0\cup \ldots \cup D_{i-1}$. Denote the smallest size of any $\tau$-dynamo of $G$ by $dyn_{\tau}(G)$. Dynamic monopolies are in fact modeling the spread of influence in social networks. The spread of innovation or a new product in a community, spread of opinion in Yes-No elections, spread of virus in the internet, spread of disease in a population are some examples of these phenomena. Obviously, if for a vertex $v$ we have $\tau(v)=deg_G(v)+1$ then $v$ should belong to any dynamic monopoly of $(G,\tau)$. We call such a vertex $v$ {\it self-opinioned} (from another interpretation it can be called {\it vaccinated vertex}). Irreversible dynamic monopolies and the equivalent concepts target set selection and conversion sets have been the subject of active research in recent years by many authors \cite{CL, CDPRS, C, DR, FKRRS, KSZ, SZ, Z1, Z2}.

\noindent In this paper by $(G,\tau)$ we mean a graph $G$ and a threshold assignment for the vertices of $G$. The average threshold of $\tau$, denoted by $\overline{\tau}$, is ${\sum}_{v\in V(G)} \tau(v)/|G|$. In Proposition \ref{middle} we show that the values of $dyn_{\tau}(G)$ over all threshold assignments with the same average threshold form a continuous set of integers. The maximum element of this set has been studied first time in \cite{KSZ}, where the following notation was introduced. Let $t$ be a non-negative rational number such that $t|G|$ is an integer, then $Dyn_t(G)$ is defined as $Dyn_t(G)=\max_{\tau: \overline{\tau}=t}~ dyn_{\tau}(G)$. The smallest size of dynamic monopolies with a given average threshold was introduced and studied in \cite{Z2}. Dynamic monopolies with given average threshold was also recently studied in \cite{CR}. In the definition of $Dyn_t(G)$, it is assumed that $t|G|$ is integer. In order to consider all values of $t$, we modify a little bit the definition. But we are forced to make a new notation, i.e. $Ldyn_t(G)$ (which stands for the largest dynamo). The formal definition is as follows.

\begin{defin}
Let $G$ be a graph and $t$ a positive number. We define $Ldyn_t(G) = \max\{dyn_{\tau}(G)| \overline{\tau}\leq t\}$. Assume that a subset $D\subseteq V(G)$ and an assignment of thresholds $\tau_0$ are such that $\bar{\tau_0}\leq t$, $|D|=dyn_{\tau_0}(G)=Ldyn_t(G)$ and $D$ is a $\tau_0$-dynamic monopoly of $(G,\tau_0)$. Then we say $(D, \tau_0)$ is a $t$-Ldynamo of $G$.
\end{defin}

\noindent $Ldyn_t(G)$ does in fact show the worst-case value of a dynamic monopoly when the average threshold is a prescribed given number. The following concept is motivated by the concept of dynamo-unbounded family of graphs, defined in \cite{Z1} concerning the smallest size of dynamic monopolies in graphs.

\begin{defin}
\noindent Let for any $n\in \Bbb{N}$, $G_n$ be a graph and $t_n$ be a number such that $0\leq t_n\leq 2\epsilon(G_n)$. We say $\{(G_n,t_n)\}_{n\in \Bbb{N}}$ is $Ldynamo$-bounded if there exists a constant $\lambda<1$ such that for any $n$, $Ldyn_{t_n}(G_n)\leq \lambda|G_n|$.
\end{defin}

\noindent Outline of the paper is as follows. In Section 2, we show that the values of $dyn_{\tau}(G)$ over all assignments $\tau$ with the same average threshold is a continuous set of integers (Proposition \ref{middle}). Then we obtain a necessary and sufficient condition for a family of graphs to be Ldynamo-bounded (Propositions \ref{sequence} and \ref{upperLdyn}).
In Section 3, it is shown that the decision problem $Ldynamo(k)$ (to be defined later) is coNP-hard for planar graphs (Theorem \ref{DynNP}) but has polynomial-time solution for forests (Theorem \ref{algorithm-forst}).

\section{Some results on $Ldyn_t(G)$}

\noindent We first show that the values of $dyn_{\tau}(G)$ over all threshold assignments $\tau$ with the same average threshold are continuous. We need the following lemma from \cite{SZ}.

\begin{lemma}\cite{SZ}\label{Lip}
Let $G$ be a graph and $\tau$ and $\tau'$ be two threshold assignments to the vertices of $G$ such that $\tau(u)=\tau'(u)$ for all vertices $u$ of $G$ except for exactly one vertex, say $v$. Then \begin{equation*}
\begin{cases}
dyn_{\tau}(G) -1 \leq dyn_{\tau'}(G) \leq dyn_{\tau}(G), & if~\tau(v)>\tau'(v),\\
dyn_{\tau}(G)\leq dyn_{\tau'}(G)\leq dyn_{\tau}(G)+1, & if~\tau(v)<\tau'(v).
\end{cases}
\end{equation*}
\end{lemma}

\noindent The continuity result is as follows.

\begin{prop}\label{middle}
Let $\tau$ and $\tau'$ be two threshold assignments for the vertices of $G$ such that $\bar{\tau}=\bar{\tau}'$. Let also $r$ be an integer such that $dyn_{\tau}(G)\leq r \leq dyn_{\tau'}(G)$. Then there exists $\tau''$ with $\bar{\tau}=\bar{\tau}''$ such that $dyn_{\tau''}(G)=r$.
\end{prop}

\noindent \begin{proof}
For any two threshold assignments $\tau$ and $\tau'$ with the same average threshold, define $\delta(\tau,\tau')={\sum}_{v:\tau(v)>\tau'(v)}(\tau(v)-\tau'(v))$. We prove the proposition by the induction on $\delta(\tau,\tau')$. If $\delta(\tau,\tau')=0$ then for any vertex $v$, $\tau(v)\leq \tau'(v)$. But the average thresholds are the same, hence $\tau=\tau'$ and the assertion is trivial. Let $k\geq 1$ and assume that the proposition holds for any two $\tau$ and $\tau'$ with the same average threshold such that $\delta(\tau,\tau')\leq k$. We prove it for $k+1$. Assume that $\tau$ and $\tau'$ are given such that $\delta(\tau,\tau')=k+1$ and $\tau\neq \tau'$. Define $W=\{v: \tau(v)>\tau'(v)\}$. Let $w\in W$. There exists a vertex $u$ such that $\tau(u)< \tau'(u)$. Since otherwise by $\bar{\tau}=\bar{\tau}'$ we would have $\tau=\tau'$. Define a new threshold $\tau''$ as follows. For any vertex $v$ with $v\not\in \{u,w\}$ set $\tau''(v)=\tau(v)$. Set also $\tau''(w)=\tau(w)-1$ and $\tau''(u)=\tau(u)+1$. We have $\delta(\tau'',\tau')\leq k$, also the average threshold of $\tau''$ is the same as that of $\tau$. So the assertion holds for $\tau''$ and $\tau'$. By Lemma \ref{Lip} we have $|dyn_{\tau}(G)-dyn_{\tau''}(G)|\leq 1$. We conclude that the assertion holds for $\tau$ and $\tau'$ too.
\end{proof}

\noindent Let $G$ be a graph and $t$ be a positive number such that $t|G|$ is integer. Let $\tau$ be any assignment with average $t$ such that $\tau(v)\leq deg_G(v)$ for any vertex $v$. Let $d_1 \leq d_2 \leq \ldots \leq d_n$ be a degree sequence of $G$ in increasing form. It was proved in \cite{KSZ} that the size of any $\tau$-dynamic monopoly of $G$ is at most $\max \{k: {\sum}_{i=1}^k (d_i+1) \leq nt\}$. The proof of this result in \cite{KSZ} shows that if we allow $\tau(v)=deg_G(v)+1$ for some vertices $v$ of $G$, then the same assertion still holds. We have the following proposition concerning this fact.

\begin{prop}
Let $t$ be a positive number. Assume that in the definition of $Ldyn_t(G)$, the threshold assignments are allowed to have self-opinioned vertices. Then $Ldyn_t(G)$ can be easily obtained by a polynomial-time algorithm.\label{self-opinion}
\end{prop}

\noindent \begin{proof}
Let $d_1 \leq d_2 \leq \ldots \leq d_n$ be a degree sequence of $G$ in increasing form. By the argument we made before Proposition \ref{self-opinion}, we have $Ldyn_{t}(G)\leq \max \{k: {\sum}_{i=1}^k (d_i+1) \leq nt\}$. Let $k_0={\max}\{k: {\sum}_{i=1}^k (d_i+1) \leq nt\}$. We obtain a threshold assignment $\tau$ as follows.
\begin{equation*}
\tau(v_i)=
\begin{cases}
deg_G(v_i)+1 & i\leq k_0,\\
0 & \text{otherwise}.
\end{cases}
\end{equation*}

\noindent Let $D=\{v_1,v_2,\ldots ,v_{k_0}\}$. It's clear that $(D,\tau)$ is a $t$-Ldynamo of $G$.

\end{proof}

\noindent In \cite{KSZ}, it was proved that there exists an infinite sequence of graphs $G_1, G_2, \ldots$ such that $|G_n|\rightarrow \infty$ and ${\lim}_{n\rightarrow \infty} Ldyn_{\epsilon(G_n)}(G_n)/|G_n| =1.$ In the following, we show that a stronger result holds. In fact we show that not only the same result holds for $Ldyn_{k\epsilon(G_n)}(G_n)$, where $k$ is any constant with $0<k\leq 2$, but also it holds for any sequence $k_n$ for which $k_n|G_n|\rightarrow \infty$. In opposite direction, Proposition \ref{upperLdyn} shows that if $k_n = {\mathcal{O}}(1/|G_n|)$ then ${\lim}_{n\rightarrow \infty} Ldyn_{k_n\epsilon(G_n)}(G_n)/|G_n| \neq 1$.

\begin{prop}\label{sequence}
There exists an infinite sequence of graphs $\{(G_n, \tau_n)\}_{n=1}^{\infty}$ satisfying $|G_n|\rightarrow \infty$ and ${\epsilon(G_n)}/{|G_n|}=o(\overline{\tau}_n)$ such that
$$\lim_{n\rightarrow \infty} \frac{Ldyn_{\overline{\tau}}(G_n)}{|G_n|} =1.$$
\end{prop}

\noindent \begin{proof}
We construct $G_n$ as follows. The vertex set of $G_n$ is disjoint union of a complete graph $K_n$ and $n$ copies of complete graphs $K_{n+1}$. There exists exactly one edge between each copy of $K_{n+1}$ and $K_n$. Set $\tau_n(v)=0$ for each vertex $v$ in $K_n$ and $\tau_n(v)=\deg(v)$ for each vertex $v$ in any copy of $K_{n+1}$. It is clear that any dynamic monopoly of $G_n$ includes at least $n$ vertices of each copy of $K_{n+1}$ and hence $Ldyn_{\overline{\tau}}(G_n)\geq n^2$. Then we have
$$1\geq\lim_{n\rightarrow \infty} \frac{Ldyn_{\overline{\tau}}(G_n)}{|G_n|} \geq\lim_{n\rightarrow \infty}\frac{n^2}{n(n+2)}=\lim_{n\rightarrow \infty}\frac{n}{n+2}=1.$$

\noindent To complete the proof we show that $\frac{\overline{\tau}_n}{|E(G_n)|/|V(G_n)|^2}\rightarrow \infty$.

\begin{eqnarray*}
\lim_{n\rightarrow \infty}\frac{\overline{\tau}_n}{|E(G_n)|/|V(G_n)|^2} & = & \lim_{n\rightarrow \infty}\frac{({n}^2+n+1)/(n+2)}{(n^2+n+n(n+{n}^2))/2{(n^2+2n)}^2} =\infty.
\end{eqnarray*}\end{proof}

\noindent Proposition \ref{sequence} shows that if $t_n$ is such that ${\epsilon(G_n)}/{|G_n|}=o(t_n)$ then $\{(G_n,t_n)\}_n$ is not necessarily Ldynamo-bounded. In opposite direction, the next proposition shows that if there exists a positive number $c$ such that $t_n$ satisfies $t_n\leq c{\epsilon(G_n)}/{|G_n|}$, then any family $\{(G_n,t_n)\}_n$ is Ldynamo-bounded.

\begin{prop}\label{upperLdyn}
Let $G$ be a graph and $c$ and $t$ be two constants such that $t\leq c\frac{\epsilon(G)}{|G|}$. Then
$$Ldyn_{t}(G)<\frac{c}{c+1}|G|.$$
\end{prop}

\noindent \begin{proof}
Let $n$ be the order of $G$. If $n<c/2$, then $\lceil cn/(c+1)\rceil=n$ and hence the inequality $Ldyn_{t}(G)<c|G|/(c+1)$ is trivial. Assume now that $n\geq c/2$. Let $d_1 \leq d_2 \leq \ldots \leq d_n$ be a degree sequence of $G$ in increasing form and set $k_0=\max \{k: {\sum}_{i=1}^k (d_i+1) \leq nt\}$. As we mentioned before, by a result from \cite{KSZ} we have $Ldyn_{t}(G)\leq k_0$. The assumption $t\leq c(\epsilon(G)/n)$ implies $nt\leq (c/2n){\sum}_{i=1}^n d_i$
and hence
${\sum}_{i=1}^{k_0} (d_i+1) \leq (c/2n){\sum}_{i=1}^n d_i$ or equivalently $(2n/c)\leq ({\sum}_{i=1}^n d_i)/{\sum}_{i=1}^{k_0} (d_i+1)$. Assume on the contrary that $k_0\geq cn/(c+1)$. Then
$$\frac{2n}{c}\leq \frac{\sum_{i=1}^{k_0} d_i + \sum_{i=k_0+1}^n d_i}{(\sum_{i=1}^{k_0} d_i)+\frac{c}{c+1}n}\leq \frac{(\sum_{i=1}^{k_0} d_i) + \frac{n^2}{c+1}}{(\sum_{i=1}^{k_0} d_i)+\frac{c}{c+1}n}.$$
Therefore
$$\frac{2n-c}{c}\sum_{i=1}^{k_0} d_i\leq \frac{n^2}{c+1}-\frac{2n^2}{c+1}.$$
The left hand side of the last inequality is positive but the other side is negative. This contradiction implies $k_0< cn/(c+1)$, as required.
\end{proof}

\section{Algorithmic results}

\noindent Algorithmic results concerning determining $dyn_{\tau}(G)$, with various types of threshold assignments such as constant thresholds or majority thresholds, were studied in \cite{CDPRS, C, DR}. In this section, we first show that it is a coNP-hard problem on planar graphs to compute the size of $D$ such that $(D,\tau)$ is a $k\epsilon(G)$-Ldynamo of $G$. Then we prove that the same problem has a polynomial-time solution for forests. The formal definition of the decision problem concerning Ldynamo is the following, where $k$ is any arbitrary but fixed real number with $0<k \leq 2$.

\noindent {\bf Name: LARGEST DYNAMIC MONOPOLY (Ldynamo(k))}

\noindent {\bf Instance:} A graph $G$ on say $n$ vertices and a positive integer $d$.

\noindent {\bf Question:} Is there an assignment of thresholds $\tau$ to the vertices of $G$ with $n\bar{\tau}=\left\lfloor nk\epsilon(G)\right\rfloor$ such that $dyn_{\tau}(G)\geq d$?\\

\noindent The following theorem shows coNP-hardness of the above problem. Recall that Vertex Cover (VC) asks for the smallest number of vertices $S$ in a graph $G$ such that $S$ covers any edge of $G$. Denote the smallest cardinality of any vertex cover of $G$ by $\beta(G)$. The problem VC is NP-complete for planar graphs \cite{GJ}.

\begin{thm}\label{DynNP}
For any fixed $k$, where $0<k\leq 2$, $Ldynamo(k)$ is coNP-hard even for planar graphs.
\end{thm}

\noindent \begin{proof}
We make a polynomial time reduction from VC (planar) to our problem. Let $<G,l>$ be an instance of VC, where $G$ is planar. Define $s=4|E(G)|\times\max\{1,1/k\}+14$ and set $p=\lfloor(ks-2)/(2-k)\rfloor-|E(G)|$. Construct a graph $H$ from $G$ as follows. To each vertex $v$ of $G$ attach a star graph $K_{1,s-1}$ in such a way that $v$ is connected to the central vertex of the star graph. Consider one of these star graphs and let $y$ be a vertex of degree one in it. Add a path $P$ of length $p-1$ starting from $y$ (see Figure \ref{fig2}). The path $P$ intersects the rest of the graph only in $y$. Call the resulting graph $H$. Since $G$ is planar, $H$ is planar too.
\begin{figure}[ht]
\begin{center}
\scalebox{1} 
{
\begin{pspicture}(0,-1.0420215)(5.01,1.0820215)
\psline[linewidth=0.03cm](3.06,0.6379785)(3.78,0.6379785)
\psline[linewidth=0.03cm](4.94,0.6379785)(4.22,0.6379785)
\pscircle[linewidth=0.03,dimen=outer](1.02,-0.022021484){1.02}
\psline[linewidth=0.03cm](2.02,0.017978515)(2.64,0.017978515)
\psline[linewidth=0.03cm](2.64,0.037978515)(3.04,0.6179785)
\psline[linewidth=0.03cm](2.64,0.017978515)(3.14,0.3379785)
\psline[linewidth=0.03cm](2.64,0.017978515)(3.16,-0.26202148)
\psline[linewidth=0.03cm](2.64,-0.0020214843)(3.08,-0.5420215)
\psdots[dotsize=0.1](3.06,0.6379785)
\psdots[dotsize=0.1](3.18,0.35797852)
\psdots[dotsize=0.1](3.18,-0.2820215)
\psdots[dotsize=0.1](3.08,-0.5420215)
\psdots[dotsize=0.1](2.62,0.017978515)
\psdots[dotsize=0.1](2.04,0.017978515)
\psdots[dotsize=0.1](4.94,0.6379785)
\psdots[dotsize=0.1](3.44,0.6379785)
\psdots[dotsize=0.1](4.54,0.6379785)
\usefont{T1}{ptm}{m}{n}
\rput(0.9914551,-0.037021484){$G$}
\psdots[dotsize=0.04](3.22,0.15797852)
\psdots[dotsize=0.04](3.24,0.037978515)
\psdots[dotsize=0.04](3.22,-0.08202148)
\psdots[dotsize=0.04](3.92,0.6379785)
\psdots[dotsize=0.04](4.02,0.6379785)
\psdots[dotsize=0.04](4.12,0.6379785)
\usefont{T1}{ptm}{m}{n}
\rput(3.0674853,0.90297854){\small $y$}
\usefont{T1}{ptm}{m}{n}
\rput(2.1474853,0.26297852){\small $x$}
\end{pspicture}
}

\caption{graph $H$}\label{fig2}
\end{center}
\end{figure}

\noindent We claim that $<G, l>$ is a yes-instance of VC if and only if $<H, l+\lfloor p/2 \rfloor+1>$ is a no-instance of $Ldynamo(k)$.
From the construction of $H$, we have $|E(H)|=|E(G)|+s+p$. Then since $p=\lfloor(ks-2)/(2-k)\rfloor-|E(G)|$ we have
\begin{eqnarray*}
 &  & p\leq(ks-2)/(2-k)-|E(G)|\\
 & \Rightarrow & 2p+2|E(G)|+2\leq k(s+p+|E(G)|)\\
 & \Rightarrow & 2p+2|E(G)|+2\leq \lfloor k|E(H)|\rfloor.
\end{eqnarray*}

\noindent Also from the value of $p$ we have
\begin{eqnarray*}
 &  & p\geq(ks-2)/(2-k)-|E(G)|-1\\
 & \Rightarrow & 2p+2|E(G)|+2+(2-k) > k(s+p+|E(G)|)\\
 & \Rightarrow & 2p+2|E(G)|+2+\lfloor 2-k \rfloor\geq \lfloor k|E(H)|\rfloor \\
  & \Rightarrow & 2p+2|E(G)|+3\geq \lfloor k|E(H)|\rfloor.
\end{eqnarray*}

\noindent Assume first that $<G,l>$ is a no-instance of VC. Then $\beta(G)\geq l+1$. We construct a threshold assignment $\tau$ for graph $H$ as follows.
\begin{equation}\label{Dynthreshold}
\tau(v)=
\begin{cases}
deg_H(v) & v\in G\cup P,\\
0 & \text{otherwise}.
\end{cases}
\end{equation}

\noindent It is easily seen that $\overline{\tau}\leq k\epsilon(H)$ and also $dyn_{\tau}(H)=\beta(G)+\lfloor p/2 \rfloor$. Therefore $<H,l+\lfloor p/2 \rfloor+1>$ is a yes-instance for $Ldynamo(k)$.

\noindent Let $<G,l>$ be a yes-instance of VC. Then $\beta(G)<l+1$. Assume that $(D,\tau)$ is a $(k\epsilon(H))$-Ldynamo of $H$. The assumption $s > 4|E(G)|+14$ implies $|D\cap (H\setminus G)|\leq \lfloor p/2 \rfloor$. From the other hand, $|D\cap G|\leq \beta(G)<l+1$. Hence $|D| < l+\lfloor p/2 \rfloor+1$. This shows that $<H,l+\lfloor p/2 \rfloor+1>$ is a no-instance for $Ldynamo(k)$. This completes the proof.
\end{proof}

\noindent In the rest of this section we obtain a polynomial-time solution for forests (Theorem \ref{algorithm-forst}). We need some prerequisites. We will make use of the concept of resistant subgraphs, defined in \cite{Z1} as follows. Given $(G,\tau)$, any induced subgraph $K\subseteq G$ is said to be a $\tau$-resistant subgraph in $G$, if for for any vertex $v \in K$ the inequality $deg_K(v)\geq deg_G(v)-\tau(v)+1$ holds, where $deg_G(v)$ is the degree of $v$ in $G$. The following proposition in \cite{Z1} shows the relation between resistant subgraphs and dynamic monopolies.

\begin{prop}(\cite{Z1})\label{resistant}
A set $D\subseteq G$ is a $\tau$-dynamo of graph $G$ if and only if $G \setminus D$ does not contain any resistant subgraph.
\end{prop}

\noindent The following lemma provides more information on resistant subgraphs which are also triangle-free.

\begin{lemma}\label{triangle-free}
Assume that $(G,\tau)$ is given. Let also $H$ be a triangle-free $\tau$-resistant subgraph in $G$ and $e=uv$ be any arbitrary edge with $u , v \in H$. Let $\tau'$ be defined as follows
\begin{equation*}
\tau'(w) =
\begin{cases}
\tau(w) & \text{if } w\notin H,\\
0 & \text{if } w\in H\setminus \{u,v'\},\\
deg_{G}(v) & \text{if } w=v,\\
deg_{G}(u) & \text{if } w=u,
\end{cases}
\end{equation*}
Then $\overline{\tau'}\leq \overline{\tau}.$
\end{lemma}

\noindent \begin{proof}
Since $H$ is triangle-free, then $|H|\geq deg_H(u)+ deg_H(v)$. From the definition of the resistant subgraphs, for any vertex $w\in H$, one has $\tau(w)\geq deg_{G\setminus H}(w)+1$. Hence the following inequalities hold.
\begin{eqnarray*}
\sum_{w\in H}\tau(w) & \geq &  \sum_{w\in H} (deg_{G\setminus H}(w)+1)  \\
{} &  \geq  &   |H| +  deg_{G\setminus H}(u)+deg_{G\setminus H}(v)\\
{} &  \geq  &    deg_H(u) + deg_H(v) + deg_{G\setminus H}(u)+deg_{G\setminus H}(v)\\
{} &  =  &  deg_G(u) + deg_G(v).
\end{eqnarray*}
It turns out that $\sum_{w\in G}\tau'(w)\leq\sum_{w\in G}\tau(w)$ and hence $\overline{\tau'}\leq \overline{\tau}$.
\end{proof}

\noindent By a {\it (zero,degree)-assignment} we mean any threshold assignment $\tau$ for the vertices of a graph $G$ such that for each vertex $v\in V(G)$, either $\tau(v)=0$ or $\tau(v)=deg_G(v)$. The following remark is useful and easy to prove. We omit its proof.

\begin{remark}\label{VC&full-empty}
Assume that $(G,\tau)$ is given where $\tau$ is (zero,degree)-assignment. Let $G_1$ be the subgraph of $G$ induced on $\{v\in G| \tau(v)=deg_G(v)\}$. Then every minimum vertex cover of $G_1$ is a minimum $\tau$-dynamo of $G$, and vice versa.
\end{remark}

\noindent The following theorem concerning (zero,degree)-assignments in forests is essential in obtaining an algorithm for $t$-Ldynamo of forests for a given $t$.

\begin{thm}\label{FM}
Let $F$ be a forest and t be a positive constant. There exists a (zero,degree)-assignment $\tau'$ such that $\overline{\tau'}\leq t$ and
$$Ldyn_{t}(F)=dyn_{\tau'}(F).$$
\end{thm}

\noindent \begin{proof}
Let $(D,\tau)$ be a $t$-Ldynamo of $F$. We prove the theorem by induction on $|D|$. Assume first that $|D|=1$. Then by Proposition \ref{resistant}, $F$ has at least one $\tau$-resistant subgraph say, $F'$. Let $u$ and $v$ be two adjacent vertices in $F'$. Let $\tau'$ be the threshold assignment constructed in Lemma \ref{triangle-free} such that $\tau'(u)=deg_F(u)$ and $\tau'(v)=deg_F(v)$. Modify $\tau'$ so that $\tau'(w)=0$ for every vertex $w\in F\setminus \{u, v\}$. It is clear that $\tau'$ is a (zero,degree)-assignment. The edge $uv$ is a $\tau'$-resistant subgraph in $F$ and hence $dyn_{\tau}(F)=Ldyn_{t}(F)=1$. This proves the induction assertion in this case.

\noindent Now assume that the assertion holds for any forest $F$ with $|D|< k$. Let $F$ be a forest with $Ldyn_{t}(F)=k$ and $D$ be a $t$-Ldynamo of $F$ with $|D|=k$. Let also $F_1$ be the largest $\tau$-resistant subgraph of $F$. For any $v\in F_1$, set $\varphi(v)=\tau(v)-deg_{F\setminus F_1}(v)$. By the definition of resistant subgraphs, $\varphi(v) > 0$. It is clear that $dyn_{\varphi}(F_1)=k$. We show that there exists a (zero,degree)-assignment $\tau'_1$ for $F_1$ such that $(D_1,\tau'_1)$ is a $\overline{\varphi}$-Ldynamo of $F_1$ with $|D_1|=k$.

\noindent Let $T$ be a connected component of $F_1$. Consider $T$ as a top-down tree, where the toppest vertex is considered as the root of $T$. Since $T$ is a $\varphi$-resistant subgraph in $F_1$, it implies that $D_1\cap T$ is not the empty set. We argue that $D_1$ can be chosen in such a way that it does not contain any vertex $w\in T$ with $\varphi(w)=1$, except possibly the root. The reason is that if $w\in D_1\cap T$ with $\varphi(w)=1$, then we replace $w$ by its nearest ancestor (with respect to the root of $T$) whose threshold is not $1$; and if there is no such ancestor then we replace $w$ by the root. Let $v\in D_1\cap T$ be the farthest vertex from the root of $T$. Let $T_v$ be the subtree of $T$ consisting of $v$ and its descendants. Obviously $T_v\cap D_1=\{v\}$.

\noindent Now we show that $T_v$ is a $\varphi$-resistant subgraph in $F_1$. For each vertex $w\in T_v\setminus\{v\}$, since $\varphi(w)\geq 1$ and $deg_{F_1\setminus T_v}(w)=0$, then $\varphi(w)\geq deg_{F_1\setminus T_v}(w)+1$. We have also $\varphi(v)\geq deg_{F_1\setminus T_v}(v)+1$. Since if $\varphi(v)=1$, then $v$ is the root of $T$ and $T_v=T$ and hence $deg_{F_1\setminus T_v}(v)=0$. And if $\varphi(v)>1$, then $deg_{F_1\setminus T_v}(v)\leq1$. This proves that $T_v$ is a $\varphi$-resistant subgraph in $F_1$. Let $v'$ be an arbitrary neighbor of $v$ in $T_v$. We construct the threshold assignment $\tau_1$ for $F_1$ as follows.
\begin{equation*}
\tau_1(w) =
\begin{cases}
\varphi(w) & \text{if } w\notin T_v,\\
0 & \text{if } w\in T_v\setminus \{v,v'\},\\
deg_{F_1}(w) & \text{if } w\in\{v,v'\}.
\end{cases}
\end{equation*}

\noindent By Lemma \ref{triangle-free}, we have $\overline{\tau_1}\leq \overline{\varphi}$. Since edge $vv'$ is a $\tau_1$-resistant subgraph in $F_1$, then $dyn_{\tau_1}(F_1)=dyn_{\varphi}(F_1)=k$ and so $D_1$ is a minimum $\tau_1$-dynamo of $F_1$. Set $F_2=F_1\setminus T_v$. Let $u$ be the parent of the vertex $v$. Construct the threshold assignment $\tau_2$ for $F_2$ as follows.
\begin{equation*}
\tau_2(w) =
\begin{cases}
\tau_1(w) & \text{if } w\in F_2\setminus\{u\},\\
\tau_1(w)-1 & \text{if } w=u .
\end{cases}
\end{equation*}

\noindent It is easily seen that the union of any $\tau_2$-dynamo of $F_2$ and $\{v\}$ is a $\tau_1$-dynamo of $F_1$ and also $D_1\setminus\{v\}$ is a $\tau_2$-dynamo of $F_2$. Hence $dyn_{\tau_2}(F_2)=dyn_{\tau_1}(F_1)-1=k-1$. Let $\varphi_2$ be any threshold assignment for $F_2$ with $\overline{\varphi_2}=\overline{\tau_2}$. Now construct the threshold assignment $\varphi_1$ for $F_1$ as follows.
\begin{equation*}
\varphi_1(w) =
\begin{cases}
\varphi_2(w) & \text{if } w\in F_2\setminus \{u\},\\
\tau_1(w) & \text{if } w\in T_v,\\
\varphi_2(w)+1 & \text{if } w=u.
\end{cases}
\end{equation*}

\noindent Because the union of any $\varphi_2$-dynamo of $F_2$ and $\{v\}$, forms a $\varphi_1$-dynamo of $F_1$ and also for any $\varphi_1$-dynamo $P$ of $F_1$, the set $P\cap F_2$ is a $\varphi_2$-dynamo of $F_2$ then $P\nsubseteq F_2$. This result and $dyn_{\tau_2}(F_2)=k-1$ imply that
$Ldyn_{\overline{\tau}_2}(F_2)=k-1$.
From the induction hypothesis there exists a (zero,degree)-assignment $\tau'_2$ for $F_2$ with $\overline{\tau'_2}\leq \overline{\tau_2}$ such that $dyn_{\tau'_2}(F_2)=k-1$. Now we construct the (zero,degree)-assignment $\tau'_1$ for $F_1$ as follows.
\begin{equation*}
\tau'_1(w) =
\begin{cases}
\tau'_2(w) & \text{if } w\in F_2\setminus \{u\},\\
\tau_1(w) & \text{if } w\in T_v,\\
\tau'_2(w)+1 & \text{if } w=u ~\text{and}~ \tau'_2(u)\neq0\\
0 & \text{if } w=u ~\text{and}~ \tau'_2(u)=0.
\end{cases}
\end{equation*}

\noindent It is easily seen that $dyn_{\tau'_1}(F_1)=k$. We finally obtain the desired (zero,degree)-assignment $\tau'$ for $F$ as follows.
\begin{equation*}
\tau'(w) =
\begin{cases}
deg_F(w) & \text{if } w\in F_1, \tau'_1(w)=deg_{F\setminus F_1}(w),\\
0 & \text{if } w\in F_1, \tau'_1(w)=0,\\
0 & \text{if } w\notin F_1.
\end{cases}
\end{equation*}
\end{proof}


\noindent In the following we show that for any forest there exists a (zero,degree)-assignment which is zero outside the vertices of a matching.

\begin{prop}\label{matching}
Let $F$ be a forest and $t$ a positive constant. Then there exists a matching $M$ such that for the (zero,degree)-assignment $\tau$ defined below, we have $\overline{\tau}\leq t$ and $Ldyn_{t}(F)=dyn_{\tau}(F)=|M|$,
\begin{equation*}
\tau(w)=
\begin{cases}
deg_F(w) & \text{if} ~ w ~ \text{is a vertex saturated by} ~ M, \\
0 & \text{otherwise.}
\end{cases}
\end{equation*}
\end{prop}
\begin{proof}
By Theorem \ref{FM}, there exists a (zero,degree)-assignment $\tau'$ such that $\overline{\tau'}\leq t$ and $Ldyn_{t}(F)=dyn_{\tau'}(F)$. Let $F_1$ be a subgraph induced on all vertices $w$, with $\tau'(w)=deg_F(w)$. Let $D$ be a minimum vertex cover of $F_1$. Remark \ref{VC&full-empty} implies that $D$ is a minimum $\tau'$-dynamic monopoly of $F$. Assume that $M$ is a maximum matching of $F_1$. We show that $M$ satisfies the conditions of the theorem. Each edge of $M$ forms a $\tau$-resistant subgraph in $F$. Hence $dyn_{\tau}(F)\geq |M|$. Using the so-called K\"{o}nig Theorem on bipartite graphs we have $|D|=|M|$. Consequently, $dyn_{\tau}(F)\geq |D|=dyn_{\tau'}(F)=Ldyn_{t}(F)$. It is easily seen that $\overline{\tau}\leq\overline{\tau'}\leq t$. The proof completes.
\end{proof}

\noindent To prove Theorem \ref{algorithm-forst}, we need the following proposition whose proof is given in the appendix.

\begin{prop}\label{flow}
Let $G$ be a bipartite graph, where each edge $e$ has a cost $c(e)\geq0$. Let also $d$ be a positive number. Then there is a polynomial time algorithm which finds a maximum matching $M$ in $G$ with $cost(M)\leq d$, where
$cost(M)={\sum}_{e\in M} c(e)$.
\end{prop}

\noindent We are ready now to present the next result.

\begin{thm}\label{algorithm-forst}
Given a forest $F$ and a positive number $t$, there exists an algorithm which computes $Ldyn_t(F)$ in polynomial-time.
\end{thm}

\noindent \begin{proof}
For each edge $e=uv$ of $F$ define $cost(e)=deg_F(u)+deg_F(v)$ and for each $S\subseteq E(F)$ define $cost(S)={\sum}_{e\in S}cost(e)$. Let $M$ be any arbitrary matching and $\tau$ be a (zero,degree)-assignment constructed from $M$ as obtained in Proposition \ref{matching}. It is easily seen that $\overline{\tau}\leq t$ if and only if $cost(M)\leq t|F|$. Now, if $M$ is a maximum matching satisfying $cost(M)\leq t|F|$, then Proposition \ref{matching} implies $Ldyn_{t}(F)=dyn_{\tau}(F)=|M|$. By Proposition \ref{flow} there is a polynomial-time algorithm which finds maximum matching $M$ in $F$ with $cost(M)\leq c$ for any value $c$. Then using Proposition \ref{matching} for given forest $F$ and constant $t$, there is a polynomial time algorithm which finds a (zero,degree)-assignment $\tau$ such that $Ldyn_{t}(F)=dyn_{\tau}(F)$. From the other side, finding a minimum vertex cover in bipartite graphs is a polynomial-time problem. Therefore using Remark \ref{VC&full-empty} a minimum $\tau$-dynamic monopoly for $F$ can be found in polynomial-time.
\end{proof}

\noindent For further researches, it would be interesting to obtain other families of graphs for which $Ldynamo(k)$ has polynomial-time solution. Also we don't know yet whether $Ldynamo(k) \in NP \cup coNP$. We guess this is not true.

\section{Appendix}

\noindent We prove Proposition \ref{flow} using the minimum cost flow algorithm. The minimum cost flow problem (MCFP) is as follows (see e.g. \cite{AMO} for details).

\noindent
Let $G=(V, E)$ be a directed network with a cost $c(i,j)\geq 0$ for any of its edges $(i,j)$. Also for any edge $(i, j)\in E$ there exists a capacity $u(i,j)\geq 0$. We associate with each vertex $i\in V$ a number $b(i)$ which indicates its source or sink depending on whether $b(i)>0$ or $b(i)<0$. The minimum cost flow problem (MCFP) requires the determination of a flow mapping $f : E\rightarrow \Bbb{R}$ with minimum cost $z(f) ={\sum}_{(i,j)\in E}c(i,j)f(i,j)$ subject to the following two conditions:

{\bf (1)} $0\leq f(i,j)\leq u(i,j)$ for all $(i,j)\in E$ (capacity restriction);

{\bf (2)} $\sum_{\{j: (i,j)\in E\}}{f(i,j)}-\sum_{\{j: (j,i)\in E\}}{f(j,i)}=b(i)$ for all $i\in V$ (demand restriction).

\noindent In \cite{AMO}, a polynomial-time algorithm is given such that determines if such a mapping $f$ exists. And in case of existence, the algorithm outputs $f$. Furthermore, if all values $u(i,j)$ and $b(i)$ are integers then the algorithm obtains an integer-valued mapping $f$. In the following we prove Proposition \ref{flow}.

\noindent {\bf Theorem.}
Let $G[X,Y]$ be a bipartite graph with $cost(ij)\geq0$ for each edge $ij\in G$ and $d$ be a positive number. Then there exists a polynomial-time algorithm which finds maximum matching $M$ in $G$ with $cost(M)\leq d$.

\noindent \begin{proof}
Construct a directed network $H$ from bipartite graph $G[X,Y]$ as follows. Add two new vertices $s$ and $t$ as the source and the sink of $H$, respectively and directed edges $(s,x)$ for each $x\in X$ and $(y,t)$ for each $y\in Y$. Make all other edges directed from $X$ to $Y$. For each edge $(i,j)$ set $u(i,j)=1$ and define $c(i,j)$ as follows.
\begin{equation*}
c(i,j)=
\begin{cases}
0 & i=s ~ \text{or} ~ j=t, \\
cost(ij) & i\in X, j\in Y.
\end{cases}
\end{equation*}

\noindent For each vertex $i\in X\cup Y$, set $b(i)=0$ and define $b(s)=-b(t)=k$, where $k$ is an arbitrary positive integer. We have now an instance of MCFP. \noindent Assume that there exists a minimum cost flow mapping for this instance (obtained by the above-mentioned algorithm of \cite{AMO}). Since $u(i,j)$ and $b(i)$ are integers then $f$ is an integer-valued mapping. Therefore $f(i,j)$ is either $0$ or $1$. Let $M$ be the set of edges $(i,j)$ with $f(i,j)=1$, where $i\in X$ and $j\in Y$. Clearly $M$ is a matching of size $k$ having $cost(M)=z(f)$, where $z(f)$ is as defined in MCFP above.

\noindent Conversely, let $M'$ be any arbitrary matching in $G$ with $|M'|=k$. We construct a flow mapping $f$ as follows.
\begin{equation*}
f(i,j)=
\begin{cases}
1 & ~~~~i\in X, j\in Y , ij\in M',\\
1 &  ~~~~i=s,   jl\in M' ~\text{for some}~ l\in Y,\\
1 &  ~~~~j=t,   li\in M' ~\text{for some}~ l\in X,\\
0 & ~~~~\text{otherwise}.
\end{cases}
\end{equation*}

\noindent The conditions of MCFP are satisfied for $f$. Also $z(f)=cost(M')$. We conclude that to obtain a matching of size $k$ with the minimum cost  is equivalent to obtain a minimum cost flow mapping for the associated MCFP instance (note that $k$ is a parameter of this instance). We conclude that in order to find a matching $M$ satisfying $cost(M)\leq d$ and with the maximum size, it is enough to run the corresponding algorithm for the above-constructed MCFP instance for each $k$, where $k$ varies from $1$ to $|G|/2$. Note that $|G|/2$ is an upper bound for the size of any matching. This completes the proof.
\end{proof}

\end{document}